\documentclass[12pt, reqno]{amsart}

\usepackage{amssymb, hyperref, graphicx, amsmath, fullpage}

\numberwithin{equation}{section}
\newtheorem{thm}{Theorem}[section]
\newtheorem{lem}[thm]{Lemma}
\newtheorem{prp}[thm]{Proposition}
\newtheorem{cor}[thm]{Corollary}

\newcommand{\R}{\mathbb R}
\newcommand{\p}{\mathbb P}
\newcommand{\E}{\mathbb E}
\newcommand{\Z}{\mathbb Z}
\newcommand{\dfn}{\stackrel{\mathrm{def}}{ = }}
\newcommand{\LCD}{{\mathrm{LCD}}}
\newcommand{\card}{{\mathrm{card}}}

\begin{document}

\title{Small ball estimates for quasi-norms}

\author{Omer Friedland}
\address{Institut de Math\'ematiques de Jussieu, Universit\'e Pierre et Marie Curie \\ 4 Place Jussieu, 75005 Paris, France}
\email{omer.friedland@imj-prg.fr}

\author{Ohad Giladi}
\address{School of Mathematical and Physical Sciences, University of Newcastle \\  Callaghan, NSW 2308, Australia}
\email{ohad.giladi@newcastle.edu.au}

\author{Olivier Gu\'edon}
\address{Laboratoire d'Analyse et Math\'ematiques Appliqu\'ees, Universit\'e Paris-Est \\ 77454 Marne-la-Vall\'ee, France}
\email{olivier.guedon@u-pem.fr}

\thanks{Part of this work was done while the second named author was visiting the University of Alberta as a PIMS postdoctoral fellow. The authors would also like to thank the referees for their valuable comments.}

\date{\today}


\maketitle

\begin{abstract}
This note contains two types of small ball estimates for random vectors in finite dimensional spaces equipped with a quasi-norm. In the first part, we obtain bounds for the small ball probability of random vectors under some smoothness assumptions on their density function. In the second part, we obtain Littlewood-Offord type estimates for quasi-norms. This generalizes results which were previously obtained in \cite{FS07, RV09}. 
\end{abstract}


\section{Introduction}

Let $E = \big(\R^n, \|\cdot\|\big)$ be an $n$-dimensional space equipped with a quasi-norm $\|\cdot\|$, and let $X$ be a random vector in $E$. The present note is concerned with small ball estimates of $X$, i.e., estimates of the form
\begin{align} \label{small ball def}
\mathbb P\big(\|X\| \le t\big) \le \varphi(t),
\end{align}
where $\varphi(t) \to 0$ as $t\to 0$. 

\smallskip

Estimates of the form \eqref{small ball def} have been studied under different assumptions on $E$ and $X$. One direction is the case when $E = \ell_2^n$, i.e., when $\|\cdot\| = |\cdot|_2$ is the Euclidean norm, and $X$ is assumed to be log-concave or, more generally, $\kappa$-concave. Recall that a log-concave vector is a vector that satisfies that for every $A,B\subseteq\R^n$ and every $\lambda \in [0,1]$, 
\begin{align*}
\mathbb P\big(X\in \lambda A + (1-\lambda) B\big) \ge \mathbb P\big(X\in A\big)^{\lambda}\cdot\mathbb P\big(X\in B\big)^{1-\lambda}. 
\end{align*}
For such vectors it was shown in \cite{Pao12} that
\begin{align*}
\p\big(|X|_2 \le \sqrt n t \big) \le \big(Ct\big)^{C'\sqrt n},
\end{align*}
and this result was later generalized in \cite{AGLLOPT12} to $\kappa$-concave vectors. 

\smallskip

Another direction which has been studied is the case when $X$ is a Gaussian vector, and $\|\cdot\|$ is a \emph{general} norm. For example, in \cite{LO05} it was shown that if $X$ is a centered Gaussian vector and $\|\cdot\|$ is a norm on $\mathbb R^n$ with unit ball $K$ such that its $n$-dimensional Gaussian measure, denoted $\gamma_n(K)$, is less than $1/2$, then
\begin{align*}
\mathbb P\big(\|X\| \le t \big) \le \big(2t\big)^{\frac{\omega^2}{4}}\gamma_n(K),
\end{align*}
where $\omega$ is the inradius of $K$. See also \cite{LS01} for an earlier survey of the subject. 

\smallskip

Finally, let us mention that small ball estimates play a r\^ole in other problems, such as invertibility of random matrices and convex geometry. See e.g. \cite{GM04, RV09, PP13}. 

\smallskip

While the above results have a more geometric flavor, in the present note we will try to present a more analytic approach. For a random vector $X$, let $\phi_X$ be its characteristic function ,i.e.,
\begin{align*}
\phi_{X}(\xi) = \mathbb E\exp\big(i\langle \xi,X\rangle\big). 
\end{align*} 

Recall the following result:
\begin{thm} \cite[Theorem 3.1]{FG11} \label{thm fourier}
Assume that $\|\cdot\|$ is a quasi-norm on $\mathbb R^n$ with unit ball $K$. Then for every $t>0$,
\begin{align} \label{not smooth}
\mathbb P\big(\|X\| \le t\big) \le |K|\left(\frac{t}{2\pi}\right)^{n}\int_{\mathbb R^n}\big|\phi_X(\xi)\big|d\xi. 
\end{align}
\end{thm}

Theorem \ref{thm fourier} says that one can obtain small ball estimates by estimating the $L_1$ norm of the characteristic function of the random vector. Moreover, one can consider a ``smoothed" version of Theorem \ref{thm fourier}: consider instead of $X$ the random vector $X + tG$, where $G$ is a standard Gaussian vector in $\mathbb R^n$ which is independent of $X$. Since $\|\cdot \|$ is a quasi-norm on $\mathbb R^n$, there exists a constant $C_K>0$ such that 
\begin{align} \label{const quasi}
\|x + y\| \le C_K(\|x\| + \|y\|), ~~ x,y \in \mathbb R^n. 
\end{align}

Therefore,
\begin{align*}
\mathbb P\big(\|X + tG\| \le 2C_Kt\big) & \ge \mathbb P\big(\|X\| \le t \wedge \|G\| \le 1\big) 
\\ & = \mathbb P\big(\|X\| \le t\big)\cdot \mathbb P\big(\|G\| \le 1\big) 
\\ & = \mathbb P\big(\|X\| \le t\big)\cdot \gamma_n(K),
\end{align*} 
where $\gamma_n(\cdot)$ is the $n$-dimensional Gaussian measure. Thus, Theorem \ref{thm fourier} implies 
\begin{align*}
\mathbb P\big(\|X\|\le t\big) \le \frac{\mathbb P\big(\|X + tG\| \le 2C_Kt\big)}{\gamma_n(K)} \le \frac{|K|}{\gamma_n(K)}\left(\frac{C_Kt}{\pi}\right)^n\int_{\mathbb R^n}\big|\phi_{X + tG}(\xi)\big|d\xi. 
\end{align*}
Using the independence of $X$ and $G$,
\begin{align} \label{small ball smoothed}
\nonumber \mathbb P\big(\|X\|\le t\big) & \le \frac{|K|}{\gamma_n(K)}\left(\frac{C_Kt}{\pi}\right)^n\int_{\mathbb R^n}\big|\phi_{X}(\xi)\big|\phi_{tG}(\xi)d\xi 
\\ & = \frac{|K|}{\gamma_n(K)}\left(C'_Kt\right)^n\int_{\mathbb R^n}\big|\phi_{X}(\xi)\big|e^{-\frac{t^2|\xi |_2^2}{2}}d\xi. 
\end{align}
Inequality \eqref{small ball smoothed} enables one to obtain small ball estimates in cases where \eqref{not smooth} cannot be applied. We use it for two different sets of examples. In the first set of examples we consider continuous random vector under certain assumptions on their characteristic functions (which are nothing but the Fourier transform of their density functions). This is discussed in Section \ref{sec cont}, Theorem \ref{thm decay}. In the second set of examples, we consider random vectors $X$ of the form
\begin{align*}
X = \sum_{i = 1}^N\alpha_i a_i,
\end{align*}
where the $a_i$ are fixed vectors in $\mathbb R^n$ and the $\alpha_i$'s are i.i.d. random variables that satisfy a certain anti-concentration condition. This problem and its applications have been studied by many authors, first in the one dimensional case (i.e., when $n=1$) and later in the multidimensional case. See \cite{FS07, RV08, RV09, TV092, TV091, TV12, Ngu121, NV13} and the reference therein for more information on this subject. In the case $E = \ell_2^n$, the problem of finding a small ball estimate have been previously considered in \cite{FS07, RV09} and is called a Littlewood-Offord type estimate. Here such an estimate is obtained for a general quasi-norm. This is discussed in Section \ref{sec lo}, Theorem \ref{thm lo}. We recall that in \cite{RV09}, Littlewood-Offord estimates were used to estimate the smallest singular value of a rectangular matrix, where the smallest singular value of a matrix $A$ is defined as $\inf_{|x|_2 =1}|Ax|_2$. It would be interesting to try and use Theorem \ref{thm lo} to estimate $\inf_{\|x\|=1}\|Ax\|$ where now $\|\cdot\|$ is a general norm, or even quasi-norm. See e.g. \cite{LL15} for some recent results in this direction.

\smallskip

\noindent{\bf Notation.} In this note $C$, $C'$, etc. always denote absolute constants. $\|\cdot\|$ denotes a quasi-norm with unit ball $K$. $|\cdot|_2$ denotes the Euclidean norm on $\R^n$. $B(x,r)$ denotes the closed ball around $x$ with radius $r$ with respect to the Euclidean norm. $\gamma_n(\cdot)$ denotes the $n$-dimensional Gaussian measure.

\section{Small ball estimates for continuous random vectors} \label{sec cont}

In this section we consider continuous random vectors, i.e., vectors with density function $f_X$. For such vectors we have 
\begin{align*}
\phi_X(\xi) = \mathbb E\exp\big(i\langle \xi, X\rangle \big) = \int_{\mathbb R^n}e^{i\langle \xi, x\rangle}f_X(x)dx = \hat f(\xi). 
\end{align*}
We can rewrite \eqref{small ball smoothed} in the following way:
\begin{align*}
\p\big(\|X\| \le t \big) \le \frac{|K|}{\gamma_n(K)}\left(C_K't\right)^n\int_{\R^n}\big|\hat f_X(\xi)\big|e^{-\frac{t^2|\xi|_2^2}{2}}d\xi. 
\end{align*}
This suggests that small ball estimates are related to weighted norms of $\hat f_X$ which are in turn known to be related to smoothness properties of $f_X$. We will study small ball estimates in terms of Sobolev norms.

\subsection{Small ball estimates and Sobolev norm}

Recall the definition of Sobolev norm: if $\mathcal F$ is the Fourier transform on $\R^n$, then 
\begin{align} \label{def sobolev}
\|f\|_{\beta,p} = \|f\|_{H_{\beta,p}(\mathbb R^n)} = \left\|\mathcal F^{-1} \left( \left(1 + |\xi|^2 \right)^{\beta/2}\hat f \right)\right\|_{L_p \left(\R^n \right)},
\end{align} 
where $p\in (1,\infty)$ and $\beta>0$.

\begin{thm} \label{thm decay}
Assume that $X$ is a random vector in $\R^n$. Assume that $1 < p\le 2$. Then for every quasi-norm $\|\cdot\|$ on $\R^n$ with unit ball $K$, 
\begin{align} \label{small ball sobolev}
\p \left(\|X\| \le t \right) \le C_K'^n\frac{|K|}{\gamma_n(K)} \|f_X\|_{\beta, p}\cdot M(\beta,p,n,t) .
\end{align}
If $pt^2\le 2$, then
\begin{align*}
M(\beta,p,n,t) \le
\begin{cases}
2^{\frac n{2p}-\frac{\beta}{2}}|\mathbb S^{n-1}|^{1/p}\Gamma\left(\frac{n-\beta p}{2}\right)^{1/p}p^{\frac{\beta}{2}-\frac{n}{2p}}\cdot t^{\beta + \frac{n}{p'}} & 2 < n-\beta p,
\\ 
2^{\frac n{2p}-\frac{\beta}{2}}|\mathbb S^{n-1}|^{1/p}\left(\log\left(\frac {2e}{pt^2}\right)\right)^{1/p}p^{\frac{\beta}{2}-\frac{n}{2p}}\cdot t^{\beta + \frac{n}{p'}} & 0 < n-\beta p \le 2,
\\
|\mathbb S^{n-1}|^{1/p}\left(\log\left(\frac {2e}{pt^2}\right)\right)^{1/p}t^n & n- \beta p \le 0,
\end{cases}
\end{align*}
where $p' = p/(p-1)$. Otherwise, if $pt^2\ge 2$, then
\begin{align*}
M(\beta,p,n,t) \le
\begin{cases}
|\mathbb S^{n-1}|^{1/p}\left(2e^{-\frac{pt^2}{18}} + \left(\frac{2}{pt^2}\right)^{\frac{n-\beta p}{2}}\Gamma\left(\frac{n-\beta p}{2}\right)\right)^{1/p}t^n & 2 \le pt^2 \le n-\beta p,
\\
3^{1/p}|\mathbb S^{n-1}|^{1/p}e^{-\frac{t^2}{18}}t^n & n-\beta p \le pt^2 \le n,
\\
|\mathbb S^{n-1}|^{1/p}\left(\frac {2n} {pt^2}\log\left(\frac{ept^2}{n}\right)\right)^{\frac{n}{2p}}t^n & n \le pt^2. 
\end{cases}
\end{align*}
\end{thm}
The main tool in the proof of Theorem \ref{thm decay} is the following lemma. 

\begin{lem} \label{lem lp norm}
Let $p \in [1,2]$. If $pt^2 \le 2$ then
\begin{align*}
\frac{\big\| \left(1 + |\xi|^2 \right)^{-\frac{\beta}{2}}e^{-\frac{t^2|\xi|^2}{2}}\big\|_{L_p \left(\R^n \right)}^p }{|\mathbb S^{n-1}|}
\le
\begin{cases}
\Gamma \left(\frac{n-\beta p}{2} \right)\left(\frac{2}{pt^2}\right)^{\frac{n-\beta p}{2}} & \beta p <n-2 \\ 
\log \left(\frac{2e}{pt^2}\right)\left(\frac{2}{pt^2}\right)^{\frac{n-\beta p}{2}}& n-2\le \beta p <n \\
\log \left(\frac{2e}{pt^2}\right)& \beta p \ge n.
\end{cases}
\end{align*}
Otherwise, if $pt^2 \ge 2$ then
\begin{align*}
\frac{\big\| \left(1 + |\xi|^2 \right)^{-\frac{\beta}{2}}e^{-\frac{t^2|\xi|^2}{2}}\big\|_{L_p \left(\R^n \right)}^p }{|\mathbb S^{n-1}|}
\le
\begin{cases}
2e^{-\frac{pt^2}{18}} + \left(\frac{2}{pt^2}\right)^{\frac{n-\beta p}{2}}\Gamma\left(\frac{n-\beta p}{2}\right) & 2 \le pt^2 \le n-\beta p,
\\
3e^{-\frac{pt^2}{18}}& n-\beta p \le pt^2 \le n,
\\
\left(\frac {2n} {pt^2}\log\left(\frac{ept^2}{n}\right)\right)^{n/2} & n \le pt^2. 
\end{cases}
\end{align*}
\end{lem}

As part of the proof of Lemma \ref{lem lp norm}, we need the following. 

\begin{prp} \label{prop x alpha}
Assume that $x\ge \alpha \ge 1$. Then
\begin{align*}
\int_x^{\infty}r^{\alpha-1}e^{-r}dr \le \frac{2^{\alpha + 1}x^{\alpha}e^{-x}}{\alpha}. 
\end{align*}
\end{prp}

\begin{proof}
We have
\begin{align} \label{iden split}
\nonumber\int_x^{\infty}r^{\alpha-1}e^{-r}dr & = e^{-x}\int_0^{\infty}(u + x)^{\alpha-1}e^{-u}du 
\\ & = e^{-x}\left[\int_0^{x}(u + x)^{\alpha-1}e^{-u}du + \int_x^{\infty}(u + x)^{\alpha-1}e^{-u}du\right]. 
\end{align}
Now, 
\begin{align*}
\int_0^x(u + x)^{\alpha-1}e^{-u}du \le \int_0^x(u + x)^{\alpha-1}du = \frac{x^{\alpha}\left(2^{\alpha}-1\right)}{\alpha} \le \frac{2^{\alpha}x^{\alpha}}{\alpha}. 
\end{align*}
For the second integral, since $x + u\le 2u$ we have
\begin{align*}
\int_x^{\infty}(u + x)^{\alpha-1}e^{-u}du \le 2^{\alpha-1}\int_x^{\infty}u^{\alpha-1}e^{-u}du. 
\end{align*}
Altogether, we get in \eqref{iden split},
\begin{align*}
\int_x^{\infty}r^{\alpha-1}e^{-r}dr \le \frac{2^{\alpha}x^{\alpha}e^{-x}}{\alpha} + 2^{\alpha-1}e^{-x}\int_x^{\infty}r^{\alpha-1}e^{-r}dr. 
\end{align*}
Since $x \ge \alpha$, we have, $2^{\alpha-1}e^{-x} \le 1/2$, which completes the proof. 
\end{proof}

We can now proceed to the proof of Lemma \ref{lem lp norm}

\begin{proof}[Proof of Lemma \ref{lem lp norm}]
To estimate the norm, notice that $ \left(1 + |\xi|^2 \right)^{-\beta/2} \le \min\left(1,|\xi|^{-\beta}\right)$, and so using polar coordinates
$$
\left\| \left(1 + |\xi|^2 \right)^{-\beta/2}e^{-\frac{t^2|\xi|_2^2}{2}}\right\|_{L_p(\R^n)}^p \le |\mathbb S^{n-1}|\int_0^{\infty}r^{n-1}\min\left(1,r^{-\beta p}\right)e^{-\frac{pt^2r^2}{2}}dr .
$$
Now,
\begin{align}
& \nonumber \int_0^{\infty}r^{n-1}\min\left(1,r^{-\beta p}\right)e^{-\frac{pt^2r^2}{2}}dr 
\\
& \nonumber \qquad = \int_0^{1}r^{n-1}e^{-\frac{pt^2r^2}{2}}dr + \int_{1}^{\infty}r^{n-1-\beta p}e^{-\frac{pt^2r^2}{2}}dr
\\ 
& \qquad \label{splitting at 1} = \int_0^{1}r^{n-1}e^{-\frac{pt^2r^2}{2}}dr + \frac 1 2 \left(\frac{2}{pt^2}\right)^{\frac{n-\beta p}{2}}\int_{\frac{pt^2}{2}}^{\infty}r^{\frac{n-\beta p}{2}-1}e^{-r}dr. 
\end{align}

\noindent{\bf Case 1: Assume $pt^2\le 2$. } To bound the first term in \eqref{splitting at 1}, use the trivial bound
\begin{align} \label{bound gamma}
\int_0^1r^{n-1}e^{-\frac{pt^2r^2}{2}}dr \le \int_0^1r^{n-1}~dr = \frac 1 n. 
\end{align}
To bound the second term, note first that
\begin{align} \label{splitting gamma}
&\int_{\frac{pt^2}2}^{\infty}r^{\frac{n-\beta p}{2}-1}e^{-r}dr = \int_{\frac{pt^2}2}^{1}r^{\frac{n-\beta p}{2}-1}e^{-r}dr + \int_{1}^{\infty}r^{\frac{n-\beta p}{2}-1}e^{-r}dr. 
\end{align}
Since $pt^2 \le 2$,
\begin{align} \label{bound int r small}
\nonumber \int_{\frac{pt^2}2}^1r^{\frac{n-\beta p}2-1}e^{-r}dr & \le \int_{\frac{pt^2}2}^1r^{\frac{n-\beta p}{2}-1}dr 
\\ & = 
\begin{cases}
\frac{2}{n-\beta p} \left(1-\left(\frac{pt^2}{2}\right)^{\frac{n-\beta p}{2}} \right) \quad & \beta p \neq n, \\
\log \left(\frac{2}{pt^2}\right) \quad & \beta p = n, \\
\end{cases}
\end{align}
and also 
\begin{align} \label{bound int r large}
\int_1^{\infty}r^{\frac{n-\beta p}2-1}e^{-r}dr \le 
\begin{cases}
1 & \frac{n-\beta p}2-1 \le 0, \\
\Gamma \left(\frac{n-\beta p}{2} \right) & \frac{n-\beta p}2-1>0. 
\end{cases}
\end{align}
Plugging \eqref{bound int r small} and \eqref{bound int r large} into \eqref{splitting gamma}, we get
\begin{multline*}
\left(\frac{2}{pt^2}\right)^{\frac{n-\beta p}{2}}\int_{\frac{pt^2}{2}}^{\infty}r^{\frac{n-\beta p}{2}-1}e^{-r}dr
\le 
\\
\begin{cases}
\frac{2}{n-\beta p} \left(\left(\frac{2}{pt^2}\right)^{\frac{n-\beta p}{2}}-1 \right) + \left(\frac{2}{pt^2}\right)^{\frac{n-\beta p}{2}}\Gamma \left(\frac{n-\beta p}{2} \right) & \beta p <n-2, 
\\ 
\frac{2}{n-\beta p} \left(\left(\frac{2}{pt^2}\right)^{\frac{n-\beta p}{2}}-1 \right) + \left(\frac{2}{pt^2}\right)^{\frac{n-\beta p}{2}} & n-2\le \beta p <n,
\\
\log \left(\frac{2e}{pt^2} \right) & \beta p = n,
\\
\frac{2}{\beta p -n} \left(1-\left(\frac{2}{pt^2}\right)^{\frac{n-\beta p}{2}} \right) + \left(\frac{2}{pt^2}\right)^{\frac{n-\beta p}{2}} & \beta p >n. 
\end{cases}
\end{multline*}
Now, if $a\ge 1$ then we have
\begin{align*}
\left|\frac{a^x-1}{x}\right| \le 
\begin{cases}
a^x \log a & 0<x\le 1, \\
a^x & x\ge 1. 
\end{cases}
\end{align*}
Thus, when $\beta p <n-2$ we have
\begin{multline*}
\frac{2}{n-\beta p} \left(\left(\frac{2}{pt^2}\right)^{\frac{n-\beta p}{2}}-1 \right) + \left(\frac{2}{pt^2}\right)^{\frac{n-\beta p}{2}}\Gamma \left(\frac{n-\beta p}{2} \right) 
\\ \le \left(\frac{2}{pt^2}\right)^{\frac{n-\beta p}{2}}\left(1 + \Gamma \left(\frac{n-\beta p}{2} \right) \right)
\le 2\left(\frac{2}{pt^2}\right)^{\frac{n-\beta p}{2}}\Gamma \left(\frac{n-\beta p}{2} \right). 
\end{multline*}
When $n-2 \le \beta p <n$ we have
\begin{multline*}
\frac{2}{n-\beta p} \left(\left(\frac{2}{pt^2}\right)^{\frac{n-\beta p}{2}}-1 \right) + \left(\frac{2}{pt^2}\right)^{\frac{n-\beta p}{2}} 
\\  \le \left(\frac{2}{pt^2}\right)^{\frac{n-\beta p}{2}}\log \left(\frac{2}{pt^2} \right) + \left(\frac{2}{pt^2}\right)^{\frac{n-\beta p}{2}}
= \left(\frac{2}{pt^2}\right)^{\frac{n-\beta p}{2}}\log \left(\frac{2e}{pt^2} \right). 
\end{multline*}
Also, when $\beta p >n$ we use the fact that when $0< a \le 1$ and $x>0$, 
\begin{align*}
\frac{1-a^x}{x} \le \log \left(\frac 1 a \right),
\end{align*}
and get
\begin{align*}
\frac{2}{\beta p-n} \left(1-\left(\frac{2}{pt^2}\right)^{\frac{n-\beta p}{2}} \right) + \left(\frac{2}{pt^2}\right)^{\frac{n-\beta p}{2}} &\le \log \left(\frac{2}{pt^2} \right) + \left(\frac{2}{pt^2}\right)^{\frac{n-\beta p}{2}} 
\\ & \le \log \left(\frac{2}{pt^2} \right) + 1
\\ & = \log \left(\frac{2e}{pt^2} \right). 
\end{align*}
Altogether,
\begin{align*}
\left(\frac{2}{pt^2}\right)^{\frac{n-\beta p}{2}}\int_{\frac{pt^2}{2}}^{\infty}r^{\frac{n-\beta p}{2}-1}e^{-r}dr \le 
\begin{cases}
2\left(\frac{2}{pt^2}\right)^{\frac{n-\beta p}{2}}\Gamma \left(\frac{n-\beta p}{2} \right) & \beta p <n-2, \\ 
\left(\frac{2}{pt^2}\right)^{\frac{n-\beta p}{2}}\log \left(\frac{2e}{pt^2} \right) & n-2\le \beta p <n, \\
\log \left(\frac{2e}{pt^2} \right)& \beta p \ge n.
\end{cases}
\end{align*}
Plugging this into \eqref{splitting at 1} and using \eqref{bound gamma}, we get
\begin{align*}
\int_0^{\infty}r^{n-1}\min\left(1,r^{-p\beta}\right)e^{-\frac{pt^2r^2}{2}}dr \le 
\begin{cases}
\left(\frac{2}{pt^2}\right)^{\frac{n-\beta p}{2}}\Gamma \left(\frac{n-\beta p}{2} \right) & \beta p <n-2 \\ 
\left(\frac{2}{pt^2}\right)^{\frac{n-\beta p}{2}}\log \left(\frac{2e}{pt^2} \right)\ & n-2\le \beta p <n \\
\log \left(\frac{2e}{pt^2}\right)& \beta p \ge n, 
\end{cases}
\end{align*}
which completes the proof in the case $pt^2 \le 2$. 

\noindent{\bf Case 2: Assume $pt^2\ge 2$. }
To estimate the first term in \eqref{splitting at 1}, we consider two different cases. If $pt^2 \ge n$, choose
\begin{align*}
r_0 = \sqrt{\frac{n}{pt^2}\log\left(\frac{pt^2}{n}\right)},
\end{align*}
and then
\begin{align*}
\int_0^1r^{n-1}e^{-\frac{pt^2r^2}{2}}dr & = \int_0^{r_0}r^{n-1}e^{-\frac{pt^2r^2}{2}}dr + \int_{r_0}^1r^{n-1}e^{-\frac{pt^2r^2}{2}}dr 
\\ & \le \int_0^{r_0}r^{n-1}~dr + \int_{r_0}^1re^{-\frac{pt^2r^2}{2}}dr
\\ & \le \frac{r_0^n}{n} + \frac{1}{pt^2}e^{-\frac{pt^2r_0^2}{2}} 
\\ & = \frac 1 n\left(\frac n {pt^2}\log\left(\frac{pt^2}{n}\right)\right)^{n/2} + \frac{1}{pt^2}\left(\frac n {pt^2}\right)^{n/2}
\\ & \le \left(\frac n {pt^2}\log\left(\frac{ept^2}{n}\right)\right)^{n/2}. 
\end{align*}
Otherwise, if $pt^2\le n$, choose
\begin{align*}
r_0 = e^{-\frac{pt^2}{n}}. 
\end{align*}
Note that we have, say, $r_0 \ge 1/3$. Then, since $1-e^{-x}\le x$,
\begin{align*}
\int_0^1r^{n-1}e^{-\frac{pt^2r^2}{2}}dr & \le \frac{r_0^n}{n} + \frac{1}{pt^2}\left(e^{-\frac{pt^2r_0^2}{2}}-e^{-\frac{pt^2}{2}}\right) 
\\ & \le \frac 1 ne^{-pt^2} + \frac {e^{-\frac{pt^2r_0^2}{2}}} {pt^2}\cdot \frac{pt^2(1-r_0^2)}{2}
\\ & \le \frac 1 n e^{-pt^2} + \frac{pt^2}{n} e^{-\frac{pt^2}{18}} 
\\ & \le 2e^{-\frac{pt^2}{18}}. 
\end{align*}
For the first term in \eqref{splitting at 1} we thus have (assuming that $n \ge 2$),
\begin{align} \label{1st term t large}
\int_0^1 r^{n-1}e^{-\frac{pt^2r^2}{2}}dr \le
\begin{cases}
\left(\frac n {pt^2}\log\left(\frac{ept^2}{n}\right)\right)^{n/2} & pt^2\ge n,
\\
2e^{-\frac{pt^2}{18}} & pt^2 \le n. 
\end{cases}
\end{align}
If $n-\beta p \le 2$ then
\begin{align} \label{n small}
\int_1^{\infty}r^{n-\beta p-1}e^{-\frac{pt^2r^2}{2}}dr \le \int_1^{\infty}re^{-\frac{pt^2r^2}{2}}dr = \frac{e^{-\frac{pt^2}{2}}}{pt^2}\le e^{-\frac{pt^2}{2}}. 
\end{align}
Otherwise, if $n-\beta p \ge 2$, then again we consider two different cases. If $pt^2\le n-\beta p$ , we have
\begin{align} \label{n large t small}
\nonumber \frac 1 2 \left(\frac{2}{pt^2}\right)^{\frac{n-\beta p}{2}}\int_{\frac{pt^2}{2}}^{\infty}r^{\frac{n-\beta p}{2}-1}e^{-r}dr & \le \frac 1 2\left(\frac{2}{pt^2}\right)^{\frac{n-\beta p}{2}}\int_{0}^{\infty}r^{\frac{n-\beta p}{2}-1}e^{-r}dr 
\\ & = \frac 1 2\left(\frac{2}{pt^2}\right)^{\frac{n-\beta p}{2}}\Gamma\left(\frac{n-\beta p}{2}\right). 
\end{align}
Otherwise, suppose that we still have $n-\beta p \ge 2$, but now $pt^2 \ge n-\beta p$. Then by Proposition \ref{prop x alpha}, we have 
\begin{align} \label{n large t large}
\frac 1 2\left(\frac 2 {pt^2}\right)^{\frac{n-\beta p}{2}}\int_{\frac{pt^2}{2}}^{\infty}r^{\frac{n-\beta p}{2}-1}e^{-r}dr \le \frac {2^{\frac{n-\beta p}{2}}e^{-\frac{pt^2}{2}}}{n-\beta p} \le 2^{\frac{n-\beta p}{2}-1}e^{-\frac{pt^2}{2}} .
\end{align}
Altogether, combining \eqref{n small}, \eqref{n large t small} and \eqref{n large t large}, we obtain
\begin{align} \label{2nd term t large}
\frac 1 2\left(\frac{2}{pt^2}\right)^{\frac{n-\beta p}{2}}\int_{\frac{pt^2}{2}}^{\infty}r^{\frac{n-\beta p}{2}-1}e^{-r}dr 
\le
\begin{cases}
2^{\frac{n-\beta p}{2}}e^{-\frac{pt^2}{2}}& 2\le n-\beta p \le pt^2,
\\
\left(\frac{2}{pt^2}\right)^{\frac{n-\beta p}{2}}\Gamma\left(\frac{n-\beta p}{2}\right) & 2 \le pt^2 \le n-\beta p ,
\\
e^{-\frac{pt^2}{2}} & n- \beta p \le2 \le pt^2. 
\end{cases}
\end{align}
Plugging \eqref{1st term t large} and \eqref{2nd term t large} into \eqref{splitting at 1} gives
\begin{multline*}
\int_0^{\infty}r^{n-1}\min\left(1, r^{-\beta p}\right)e^{-\frac{pt^2r^2}{2}}dr 
\le 
\\
\begin{cases}
2e^{-\frac{pt^2}{18}} + \left(\frac{2}{pt^2}\right)^{\frac{n-\beta p}{2}}\Gamma\left(\frac{n-\beta p}{2}\right) & 2 \le pt^2 \le n-\beta p,
\\
2e^{-\frac{pt^2}{18}} + 2^{\frac{n-\beta p}{2}}e^{-\frac{pt^2}{2}}& 2 \le n-\beta p \le pt^2 \le n,
\\
\left(\frac n {pt^2}\log\left(\frac{ept^2}{n}\right)\right)^{n/2} + 2^{\frac{n-\beta p}{2}}e^{-\frac{pt^2}{2}}& 2\le n-\beta p \le n \le pt^2,
\\
2e^{-\frac{pt^2}{18}} + e^{-\frac{pt^2}{2}} & n- \beta p \le 2 \le pt^2 \le n,
\\ \left(\frac n {pt^2}\log\left(\frac{ept^2}{n}\right)\right)^{n/2} + e^{-\frac{pt^2}{2}} & n- \beta p \le 2 \le n \le pt^2 .
\end{cases}
\end{multline*}
In order to simplify the last expression, first notice that when $n \le pt^2$, we have
\begin{align*}
e^{-\frac{pt^2}{2}} \le \left(\frac n {pt^2}\log\left(\frac{ept^2}{n}\right)\right)^{n/2}. 
\end{align*}
Also, we have that whenever $pt^2\ge n-\beta p \ge 2$, since we have that $1- \log 2 >1/4$ we get the following estimate,
\begin{align*}
2^{\frac{n-\beta p}{2}}e^{-\frac{pt^2}{2}} \le e^{-\frac{pt^2}{2}\left(1 - \log 2\right)} \le e^{-\frac{pt^2}{8}} \le e^{-\frac{pt^2}{18}}. 
\end{align*}
Hence, we conclude that,
\begin{align*}
\int_0^{\infty}r^{n-1}\min\left(1, r^{-\beta p}\right)e^{-\frac{pt^2r^2}{2}}dr 
\le
\begin{cases}
2e^{-\frac{pt^2}{18}} + \left(\frac{2}{pt^2}\right)^{\frac{n-\beta p}{2}}\Gamma\left(\frac{n-\beta p}{2}\right) & 2 \le pt^2 \le n-\beta p,
\\
3e^{-\frac{pt^2}{18}}& n-\beta p \le pt^2 \le n,
\\
\left(\frac {2n} {pt^2}\log\left(\frac{ept^2}{n}\right)\right)^{n/2} & n \le pt^2. 
\end{cases}
\end{align*}
This completes the proof of Lemma \ref{lem lp norm}.
\end{proof}

We are now in a position to prove Theorem \ref{thm decay}. 

\begin{proof}[Proof of Theorem \ref{thm decay}]
we have,
\begin{align*}
\int_{\R^n}\big|\hat f_X (\xi )\big|e^{-\frac{t^2|\xi|_2^2}{2}}d\xi \le 
\left\| \left(1 + |\xi|^2 \right)^{\beta/2}\hat f_X\right\|_{L_{p'}(\R^n)} ~ \left\| \left(1 + |\xi|^2 \right)^{-\beta/2}e^{-\frac{t^2|\xi|_2^2}{2}}\right\|_{L_p(\R^n)}. 
\end{align*}
Since $1< p\le 2$, $\mathcal F:L_{p} \to L_{p'}$ is bounded with norm 1. Hence,
\begin{eqnarray*}
\left\| \left(1 + |\xi|^2 \right)^{\beta/2}\hat f_X\right\|_{L_{p'}(\R^n)} & = & \left\|\mathcal F \left(\mathcal F^{-1} \left( \left(1 + |\xi|^2 \right)^{\beta/2}\hat f_X \right) \right)\right\|_{L_{p'}(\R^n)} 
\\ & \le & \left\|\mathcal F^{-1} \left( \left(1 + |\xi|^2 \right)^{\beta/2}\hat f_X \right)\right\|_{L_p(\R^n)} 
\\ & \stackrel{\eqref{def sobolev}}{ = } & \|f_X\|_{\beta,p}. 
\end{eqnarray*}
Altogether,
$$
\int_{\R^n}\big | \hat f_X (\xi)\big |e^{-\frac{t^2|\xi|_2^2}{2}}d\xi \le\left\|\ \left(1 + |\xi|^2 \right)^{-\beta/2}e^{-\frac{t^2|\xi|_2^2}{2}}\right\|_{L_p(\R^n)}\|f_X\|_{\beta,p}. 
$$
Now use Lemma \ref{lem lp norm} to complete the proof.
\end{proof}

\smallskip

\subsection{Sobolev Embeddings and Theorem \ref{thm decay}}

We did not study the sharpness of Theorem \ref{thm decay}. In some cases, Sobolev Embedding Theorems can imply simpler proofs and better dependence in $t$. We are grateful for the referee who pointed this to us.

\noindent{\bf The case $n-\beta p <0$.}
In this case we can write
\begin{align}\label{beta p large}
\mathbb P\big(\|X\| \le t\big) = \int_{tK}f_X(x)dx \le |K|t^n\|f_X\|_{\infty}.
\end{align}
Since we assumed in Theorem \ref{thm decay} that $p\le 2$, we have just as in the proof of Theorem \ref{thm decay}
\begin{align}\label{bound sup}
\nonumber \|f_X\|_{\infty} \le  \int_{\mathbb R^n}|\hat f_X(\xi)|d\xi & \le \left\|\left(1+|\xi|^2\right)^{-\frac{\beta}{2}}\right\|_{L_p(\mathbb R^n)}\left\|\left(1+|\xi|^2\right)^{\frac{\beta}{2}}\hat f_X(\xi)\right\|_{L_{p'}(\R^n)}
\\ & \le \left\|\left(1+|\xi|^2\right)^{-\frac{\beta}{2}}\right\|_{L_p(\mathbb R^n)}\|f_X\|_{\beta,p}.
\end{align}
Now,
\begin{align}\label{bound int potential}
\nonumber \left\|\left(1+|\xi|^2\right)^{-\frac{\beta}{2}}\right\|_{L_p(\mathbb R^n)}  & = \left(\int_{\R^n}\frac{d\xi}{\left(1+|\xi|^2\right)^{\frac{\beta p}{2}}}\right)^{\frac 1 p}
\\ \nonumber &  = |\mathbb S^{n-1}|^{\frac 1 p}\left(\int_0^{\infty}\frac{r^{n-1}dr}{\left(1+r^2\right)^{\frac{\beta p}{2}}}\right)^{\frac 1 p}
\\ & \le |\mathbb S^{n-1}|\left(\frac 1 {\beta p-n}+\frac 1 n\right)^{\frac 1 p}.
\end{align}
Plugging \eqref{bound sup} and \eqref{bound int potential} into \eqref{beta p large}, we get
\begin{align*}
\mathbb P\big(\|X\| \le t\big) \le |\mathbb S^{n-1}|\left(\frac 1 {\beta p-n}+\frac 1 n\right)^{\frac 1 p}|K|t^{n}\|f_X\|_{\beta,p}.
\end{align*}
While the bound gives a better dependence on $t$ when $t$ is small (as it removes the $\log$ term), its dependence on the other parameters can be worse as the implied constant tends to infinity as $\beta p \to n$.

\noindent{\bf The case $n-\beta p > 0$.}
The Sobolev Embedding Theorem (see e.g. \cite[Ch. 9]{Bre11} for the case where $\beta$ is an integer)
\begin{align*}
\|f\|_{L_q(\R^n)} \le C(\beta,n)\|f\|_{\beta,p},
\end{align*}
where $q = \frac{n p }{n- \beta p}$. Thus, we have
\begin{align}\label{beta p small}
\nonumber \mathbb P\big(\|X\| \le t\big) & = \int_{tK}f_X(x)dx \le |tK|^{\frac 1 {q'}}\|f_X\|_{L_q(\R^n)} 
\\ & \le C(\beta,n)|K|^{\frac 1 {q'}}t^{\beta+\frac n {p'}}\|f_X\|_{\beta,p},
\end{align}
Note that once again, \eqref{beta p small} gives a better dependence on $t$ for small values of $t$. However, the term $|K|^{\frac 1 {q'}}$ might be worse than the term that appears in Theorem \ref{thm decay}. 

\noindent{\bf The case $n-\beta p =0$.} 
In this case we have 
\begin{align*}
\|f\|_{L_q(\R^n)} \le C(\beta,n,q)\|f\|_{\beta,p}.
\end{align*}
where now $q\ge p$ and $C(\beta,n,q) \to \infty$ as $q\to \infty$. As before, using the Sobolev Embedding Theorem, we get 
\begin{align*}
\mathbb P\big(\|X\| \le t\big) \le C(\beta,n,q)|K|^{\frac 1 {q'}}t^{\beta+\frac n {p'}}\|f_X\|_{\beta,p},
\end{align*}
which gives a better dependence in $t$ when $t$ is small, but possibly a worse dependence on the other parameters.

\section{Littlewood-Offord type estimates} \label{sec lo}

Let $a_1, \ldots, a_N$ be (deterministic) vectors in $\R^n$, and denote by $A$ the $N\times n$ matrix whose rows are $a_1, \ldots, a_N$. Let $\delta_1,\dots, \delta_N$ be i.i.d. random variables for which there exists $b \in (0,1)$ such that
\begin{align} \label{eq:small ball xi}
\sup_{x \in \R} \p \left( | \delta_i - x| \le 1 \right) \le 1-b . 
\end{align}
Now, consider the random vector
\begin{align} \label{def x}
X = \sum_{k = 1}^N\delta_k a_k. 
\end{align}
As in \cite{FS07, RV09}, the small ball estimate of $X$ involves the least common denominator of the matrix $A$. Thus, for $\alpha > 0$ and $\gamma\in \left(0,1 \right)$, define
\begin{align} \label{def lcd}
\LCD_{\alpha, \gamma} \left(A \right) \stackrel{\mathrm{def}}{ = } \inf \left\{ |\theta|_2 : \theta \in \R^n, d_2\left(A \theta, \Z^n \right) < \min\left(\gamma |A \theta| _2, \alpha \right) \right\}. 
\end{align}
Recall that $|\cdot|_2$ denotes the Euclidean norm.

\begin{thm} \label{thm lo}
Let $X$ be defined as in \eqref{def x}, and assume that the $N\times n$ matrix $A$ satisfies $|A\theta|_2 \ge |\theta|_2$ for all $\theta$ in $\R^n$. Assume also that $t \ge \frac{\sqrt n}{\mathrm{LCD}_{\alpha,\gamma}(A)}$. Then
$$
\p \left(\|X\| \le t \right) \le \frac{|K|}{\gamma_n \left(K \right)} \left(\frac{C_K}{\pi}\right)^n \left( \left(\frac{t}{\gamma \sqrt b} \right)^n + \exp \left(-b \alpha^2 \right) \right),
$$
where $C_K$ is again the quasi-norm constant from \eqref{const quasi}. 
In particular, for any $p > 0$, if we let $|x|_p = \left(\sum_{j=1}^n|x_j|^p\right)^{1/p}$ for $x = (x_1,\dots,x_n)\in \R^n$, then
$$
\p \left( |X|_p \le t n^{1/p} \right) \le \left(C\cdot C_p\right)^n \left( \left(\frac{t}{\gamma \sqrt b} \right)^n + \exp \left(-b \alpha^2 \right) \right) ,
$$
where $C_p = \min\left\{2^{1/p-1},1\right\}$. 
\end{thm}

The first step of the proof is to estimate the small ball probability using the integer structure of the vectors $a_i$. To do that, for a given $\theta \in \R^n$, define
\begin{align} \label{def f}
f(\theta) \stackrel{\mathrm{def}}{ = } \inf_{m\in \mathbb Z^N}\left| \frac z t A\theta-m\right|_2. 
\end{align}

\begin{lem}[Small ball estimate in terms of integer structure] \label{integ struc}
Let $X$ be a random vector as in \eqref{def x} and let $t > 0$. Then
$$
\p \left(\|X\| \le t \right) \le \frac{|K|}{\gamma_n(K)}\left(C'_Kt\right)^n\cdot \sup_{z \ge \frac{1}{2 \pi}}\int_{\R^n}e^{-4bf(\theta)^2-{|\theta|_2^2}/{2}}d\theta. 
$$
\end{lem}

\begin{proof}
By \eqref{small ball smoothed} we have 
$$
\p \left( \|X\| \le t \right) \le \frac{|K|}{\gamma_n \left(K \right)} \left(C'_Kt\right)^n\int_{\R^n} \big|\phi_{X}(\xi)\big|e^{-\frac{t^2|\xi|_2^2}{2}} d\xi. 
$$
Setting $\theta = t\xi$, 
\begin{align} \label{change var}
t^n \int_{\R^n} \big| \phi_{X} (\xi) \big|e^{-\frac{t^2|\xi|_2^2}{2}} d\xi = \int_{\R^n} \big| \phi_{X} ({\theta}/{t}) \big| e^{-\frac{ |\theta|_2^2}{2}} d\theta. 
\end{align}
Using the definition of $X$, and the independence of $\delta_1,\ldots,\delta_N$, we have 
\begin{multline} \label{char prod}
\left| \phi_{X} \left({\theta}/{t}\right) \right| = \mathbb E \exp\left(i\left\langle \sum_{i = 1}^N\delta_i a_i, \theta/t\right\rangle\right) 
\\ = \prod_{k = 1}^N \mathbb E\exp\left(i\delta_k \frac{\langle a_k,\theta\rangle}{t} \right) = \prod_{k = 1}^N \left| \phi_\delta\left( \frac{\langle \theta, a_k \rangle}{t} \right) \right|,
\end{multline}
where $\delta$ is an independent copy of $\delta_1,\dots,\delta_N$. In order to estimate the right side of \eqref{char prod}, follow the conditioning argument that was used in \cite{FS07, RV09}. Let $\delta'$ be an independent copy of $\delta$, and denote by $\bar\delta$ the symmetric random variable $\delta-\delta'$. We have, $|\phi_{\delta} \left(\xi \right)|^2 = \E\cos \left(\xi\bar\delta \right)$. Using the inequality $|x|\le \exp \left(- \left(1-x^2 \right)/2 \right)$, which is valid for all $x\in\R$, we obtain 
\begin{align} \label{bound char delta}
|\phi_{\delta} \left(\xi \right)| \le \exp \left( - \frac{\left(1-\E\cos \left(\xi\bar\delta \right) \right)}{2} \right). 
\end{align}
By assumption \eqref{eq:small ball xi} it follows that $\p \left(|\bar\delta| \ge 1 \right)\ge b$. Therefore, by conditioning on $\bar\delta$, we get
\begin{align*}
1-\E\cos \left(\xi\bar\delta \right) & \ge \p \left(|\bar\delta|\ge 1 \right) \cdot \E \left( 1 - \cos \left(\xi \bar \delta \right) \Big| | \bar \delta | \ge 1 \right) 
\\ & \ge b\E \left( 1 - \cos \left(\xi \bar \delta \right) \Big| | \bar \delta | \ge 1 \right). 
\end{align*}
By the fact that $1 - \cos \theta \geq \frac{2}{\pi^2} \theta^2$, for any $|\theta|\le \pi$, we have for any $\theta\in\R$, 
$$
1 - \cos \theta \ge\frac{2}{\pi^2} \min_{m \in \Z} |\theta - 2\pi m|^2. 
$$
Hence,
\begin{align*}
1-\E\cos \left(\xi\bar\delta \right) & \ge \frac {2b} {\pi^2} \cdot \E\left( \min_{m \in \Z} \big| {\xi\bar\delta} - 2\pi m \big|^2 \Big| \big|{\bar\delta} \big| \ge 1 \right) 
\\ & = 8b \cdot \E\left( \min_{m \in \Z} \big| {\xi\bar\delta} - m \big|^2 \Big| \big| {\bar\delta} \big| \ge 1/2\pi \right). 
\end{align*}
Plugging this into \eqref{bound char delta} gives
\begin{align} \label{bound with exp}
\big|\phi_{\delta}(\xi)\big| \le \exp\left(-4b \mathbb E\left( \min_{m \in \Z} \big| {\xi\bar\delta} - m \big|^2 \Big| \big| {\bar\delta} \big| \ge 1/2\pi \right)\right). 
\end{align}
Replacing the conditional expectation with supremum over all the possible values $z \ge 1/2\pi$ and using Jensen's inequality, we get
\begin{eqnarray}
\nonumber && \int_{\R^n} \left|\phi_X\left({\theta}/{t}\right)\right|e^{-{|\theta|_2^2}/{2}}d\theta 
\\ \nonumber && \qquad \stackrel{\eqref{char prod}}{ = } \int_{\R^n}\prod_{k = 1}^N\left|\phi_{\delta}\left(\frac{\langle \theta,a_k\rangle}{t}\right)\right|e^{-{|\theta|_2^2}/{2}}d\theta \\
\nonumber && \qquad \stackrel{\eqref{bound with exp}}{\le} \int_{\R^n} \exp \left( -4b\cdot \E \left( \sum_{k = 1}^N \min_{m \in \Z} \left| {\frac{\langle \theta, a_k \rangle}{t} \bar\delta} - m \right|^2 \Bigg| \big| {\bar\delta} \big| \ge 1/2\pi \right) -{|\theta|_2^2}/{2}\right) d\theta \\
\nonumber && \qquad \le \E \left[ \int_{\R^n} \exp \left( -4b \min_{m \in \Z^N} \left| {\frac{\bar\delta}{t} A\theta} - m \right|_2^2 -{|\theta|_2^2}/{2}\right) d\theta \Bigg| \big| {\bar\delta} \big| \ge 1/2\pi \right] \\
\nonumber & & \qquad \le \sup_{z \ge 1/2\pi} \int_{\R^n} \exp \left( -4b f \left(\theta \right)^2-{|\theta|_2^2}/{2} \right)d\theta .
\end{eqnarray}
Using \eqref{change var} the result follows. 
\end{proof}

Define the set 
$$
T_s \stackrel{\mathrm{def}}{ = } \left\{\theta\in \R^n : f(\theta) \le s\right\}. 
$$
The next step in the proof is to rewrite the integral that appears in Lemma \ref{integ struc} in the following way:
\begin{multline} \label{int level set}
\int_{\R^n} \exp \left( -4b f \left(\theta \right)^2 \right)\exp \left( -|\theta|_2^2 /2 \right) d\theta 
\\ = \int_{\R^n} \int_{s \ge f \left(\theta \right)} 8bs \exp \left( -4b s^2 \right) ds \exp \left(-|\theta|_2^2 /2 \right) d\theta
\\ = \left(2\pi \right)^{n/2} \int_0^{\infty} 8bs \exp \left(-4bs^2 \right) \gamma_n \left(T_s\right) ds,
\end{multline}
which means that we have to bound $\gamma_n(T_s)$. To do that, we start with the following covering lemma. 

\begin{lem}[Covering of $T_s$] \label{lem covering}
Let $\alpha > 0$ and $\gamma \in (0,1)$. Assume that $t \ge \frac{\sqrt n}{\mathrm{LCD}_{\alpha, \gamma}(A)}$. Assume also that $|A\theta|_2 \ge |\theta|_2$ for all $\theta \in \R^n$. If $0 \le s \le \alpha/2$, then there exist vectors $\{x_i\}_{i \in I}\subseteq \R^n$ such that
\begin{align} \label{good covering}
T_s \subseteq \bigcup_{i\in I}B(x_i, r) \text{ and } |x_i-x_{i'}|_2 \ge R, \forall i \neq i',
\end{align}
where $r = \frac{2st}{\gamma z}$ and $R = \frac{\sqrt n}{z}$. Moreover, for any $j \ge 1$, 
\begin{align} \label{small shells}
\card \left( \left\{ i \in I : j R \le |x_i|_2 < \left( j + 1 \right)R \right\} \right) \le n2^n \left(j + 1 \right)^{n-1}. 
\end{align}
\end{lem}

\begin{proof}
Let $\theta_1, \theta_2 \in T_s$. By \eqref{def f}, there exists $p_1, p_2 \in \Z^N$ such that
$$
\left| \frac{z}{t} A \theta_1 - p_1\right| \le s \quad \hbox{and} \quad \left| \frac{z}{t} A \theta_2 - p_2\right| \le s. 
$$
By the triangle inequality, 
$$
\left| \frac{z}{t} A \left(\theta_1-\theta_2 \right) - \left(p_1-p_2 \right) \right| \le 2s ,
$$
which means that $d_2 \left(A\tau, \Z^N \right) \le 2s \le \alpha$, where $\tau = z\left(\theta_1 - \theta_2 \right)/t$. By \eqref{def lcd} this implies that either
$$
|\tau|_2 \ge \LCD_{\alpha, \gamma} \left(A \right) ,
$$
or
$$
\alpha \ge 2 s \ge d_2\left(A\tau, \Z^N \right) \ge \min\left( \gamma |A \tau|_2, \alpha \right) = \gamma |A\tau|_2. 
$$ 
By the assumptions that $|A \tau|_2 \ge |\tau|_2$ and $\LCD_{\alpha, \gamma} \left(A \right) \ge \sqrt n / t$, we conclude that
$$
\hbox{either } \quad |\theta_1 - \theta_2|_2 \ge \frac{\sqrt n}{z} = : R \quad \hbox{ or } \quad |\theta_1 - \theta_2|_2 \le \frac{2s t}{\gamma z} = : r. 
$$
Hence, $T_s$ can be covered by a union of Euclidean balls of radius $r$ whose centers are $R$-separated, which proves \eqref{good covering}. Next, for $j\ge 1$, let 
$$
M_j \stackrel{\mathrm{def}}{ = } \card \left( \left\{i \in I : jR \le |x_i|_2 \le (j + 1)R \right\} \right). 
$$
To estimate $M_j$, use a well-known volumetric argument. Indeed, since $\{x_i \}_{i \in I}$ are $R$-separated, we know that the Euclidean balls $B \left(x_i, R/2 \right)$ are disjoint and contained in the shell 
$$
 \left\{y \in \R^n : \left(j-1/2 \right)R \le |y|_2 \le \left(j + 3/2 \right) R\right\}. 
$$
Hence, taking the volume, 
\begin{align*}
M_j \left( \frac{R}{2} \right)^n & \le R^n \left( \left(j + 3/2 \right)^n - \left(j-1/2 \right)^n \right) 
\\ & = R^n \left(j + 1/2 \right)^n \left( \left(1 + \frac{2}{2j + 1} \right)^n - \left(1- \frac{2}{2j + 1} \right)^n \right). 
\end{align*}
Since for every $x \in \left(0,1 \right)$, we have $ \left(1 + x \right)^n - \left(1-x \right)^n \le 2nx \left(1 + x \right)^{n-1}$, we conclude that 
$$
M_j \le n2^n \left(j + 1 \right)^{n-1}. 
$$
This completes the proof.
\end{proof}

Using the covering lemma, we can now prove the required volume estimate. 

\begin{cor} \label{cor measure}
Let $r$ and $R$ be as in Lemma \ref{lem covering}. If $R \ge 2 r$, then
$$
\gamma_n \left(T_s \right) \le \left(\frac{Cr}{R} \right)^n = \left( \frac{2 C t s}{\gamma \sqrt n} \right)^n. 
$$
\end{cor}

\begin{proof}
Let $y \in \R^n$, we have 
$$
\gamma_n \left( B \left(y,r \right) \right) = \frac{1}{ \left(2\pi \right)^{n/2}} \int_{|y-x|_2 \le r} e^{-\frac{|x|_2^2}{2}} dx. 
$$
Since $|x|_2^2 + |x-y|_2^2 = \frac{1}{2} \left(|y|_2^2 + |2x-y|_2^2 \right) \ge \frac{1}{2} |y|_2^2$,
$$
\gamma_n \left( B \left(y,r \right) \right) \le \frac{1}{ \left(2\pi \right)^{n/2}} e^{- \frac{|y|_2^2}{4}} \int_{|y-x|_2 \le r} e^{\frac{|y-x|_2^2}{2}} dx. 
$$
Therefore, if $|y|_2 \ge R\ge 2r$,
\begin{align} \label{eq:gaussian} 
\nonumber \gamma_n \left( B \left(y,r \right) \right) & \le \frac{1}{ \left(2\pi \right)^{n/2}} \exp \left(- \frac{|y|_2^2}{4} \right) e^{r^2/2} |B \left(0,r \right)| 
\\ & \le \frac{1}{ \left(2\pi \right)^{n/2}} \exp \left( - \frac{|y|_2^2}{8} \right) |B \left(0,r \right)|. 
\end{align}
Assume that $s$ is such that $r \le R/2$, i.e. $4 t s\le \gamma \sqrt n$, then by \eqref{good covering}
$$
\gamma_n \left(T_s \right) \le \sum_{i \in I} \gamma_n \left( B \left(x_i, r \right) \right) \le \sum_{j = 0}^{ \infty} \sum_{i \in I : jR \le |x_i|_2 < \left(j + 1 \right)R} \gamma_n \left( B \left(x_i, r \right) \right). 
$$
Also, for $j \ge 1$, we have by \eqref{small shells}
$$
\card \left( \left\{i \in I :j R \le |x_i|_2 < \left(j + 1 \right)R \right\} \right) \le C^n j^{n-1}. 
$$
By $ \left(\ref{eq:gaussian} \right)$,
$$
\gamma_n \left( B \left(x_i, r \right) \right) \le \frac{1}{ \left(2\pi \right)^{n/2}} \exp \left( - \frac{j^2 R^2}{8} \right) |B (0,r )|. 
$$
Hence,
\begin{align} \label{bound gamma t}
\nonumber\gamma_n \left(T_s \right) & \le \gamma_n \left( B \left(0,r \right) \right) + \sum_{j = 1}^{ \infty} \left(\frac{C^2}{2 \pi} \right)^{n/2} j^{n-1}\exp \left( - \frac{j^2 R^2}{8} \right) |B \left(0,r \right)| \\
& \le \frac{ |B \left(0,r \right)| }{ \left(2\pi \right)^{n/2}} \left(1 + C^n \sum_{j = 1}^{ \infty} j^{n-1} \exp \left( - \frac{j^2 R^2}{8} \right) \right). 
\end{align}
The function $v \mapsto v^{n-1} e^{-v^2 R^2 /8}$ is decreasing for $v \ge 2 \sqrt n / R$. By comparing series with integrals, 
\begin{align*}
\sum_{j = 1}^{\infty} j^{n-1} \exp \left( - \frac{j^2 R^2}{8} \right) & \le \left( \frac{2 \sqrt n}{R} \right)^n + \int_0^{\infty} v^{n-1} e^{-v^2 R^2 /8} dv \\
& = \left( \frac{2 \sqrt n}{R} \right)^n + \frac{8^{n/2}}{R^n} \int_0^{\infty} u^{\frac{n-1}{2} - 1} e^{-u} du \le \left( \frac{C n^{1/2}}{R} \right)^{n}. 
\end{align*}
Since $z \ge 1/ 2\pi$, we have $R \le 2\pi \sqrt n$, so that
$$
\left(1 + C^n \sum_{j = 1}^{\infty} j^{n-1} \exp \left( - \frac{j^2 R^2}{8} \right) \right) \le \left( \frac{C_1 n^{1/2}}{R}\right)^{n}. 
$$
Moreover, it is well-known that $|B \left(0,r \right)| \le C_2^n n^{-n/2} r^n$ which implies by \eqref{bound gamma t} that
$$
\gamma_n \left(T_s \right) \le \left(\frac{Cr}{R} \right)^n = \left( \frac{2 C t s}{\gamma \sqrt n} \right)^n,
$$
which completes the proof.
\end{proof}

We are now in a position to prove Theorem \ref{thm lo}. 

\begin{proof}[Proof of Theorem \ref{thm lo}]
By Lemma \ref{integ struc} and \eqref{int level set}, to have a small ball estimate it is enough to evaluate the integral
$$
\int_0^{\infty}8bs\exp\left(-4bs^2\right)\gamma_n\left(T_s\right)ds. 
$$
We have,
\begin{align*}
&\int_0^{\infty}8bs\exp\left(-4bs^2\right)\gamma_n\left(T_s\right)ds
\\ & \quad \quad = \int_0^{\alpha/2}8bs\exp\left(-4bs^2\right)\gamma_n\left(T_s\right)ds + \int_{\alpha/2}^{\infty}8bs\exp\left(-4bs^2\right)\gamma_n\left(T_s\right)ds
\\ & \quad \quad \le \int_0^{\alpha/2}8bs\exp\left(-4bs^2\right)\gamma_n\left(T_s\right)ds + \exp\left(-b\alpha^2\right). 
\end{align*}
Assume first that $\alpha \le \frac{\gamma \sqrt n}{2t}$ so that for any $s \le \alpha/2$ we have $R \ge 2r$. By Corollary \ref{cor measure},
$$
\gamma_n \left(T_s\right) \le \left( \frac{2Cts}{\gamma \sqrt n} \right)^n ,
$$
and so
\begin{align*}
\int_0^{\alpha/2} 8bs \exp \left(-4bs^2 \right) \gamma_n \left(T_s\right) ds & \le\int_0^{\alpha/2} 8bs \exp \left(-4bs^2 \right) \left( \frac{2 C t s}{\gamma \sqrt n} \right)^n ds \\
& \le 8b \left(\frac{2Ct}{\gamma \sqrt n} \right)^n \int_0^{ \infty} s^{n + 1} e^{-4bs^2} ds \\
& = \left(\frac{Ct}{\gamma \sqrt b \sqrt n} \right)^n \int_0^{ \infty} u^{n/2} e^{-u} du \\
& \le \left(\frac{C' t}{\gamma \sqrt b} \right)^n. 
\end{align*}
Assume otherwise that $\alpha \ge  \frac{\gamma \sqrt n}{2t} \dfn \alpha_0$. Then, as before,
$$
\int_0^{\infty} 8bs \exp \left(-4bs^2 \right) \gamma_n \left(T_s\right) ds \le \int_0^{\alpha_0/2} 8bs \exp \left(-4bs^2 \right) \gamma_n \left(T_s\right) ds + \exp \left(-b \alpha_0^2 \right). 
$$
For $s \le \alpha_0/2$ we do exactly the same computation as in the first case and obtain
$$
\int_0^{\infty} 8bs \exp \left(-4bs^2 \right) \gamma_n \left(T_s\right) ds \le \left(\frac{C' t}{\gamma \sqrt b} \right)^n + \exp \left(-b \alpha_0^2 \right). 
$$
In this case, we also have 
$$
\exp \left(-b \alpha_0^2 \right) = \exp \left(-\frac{b \gamma^2 n}{4 t^2} \right) \le \left(\frac{C t}{\gamma \sqrt b} \right)^n. 
$$
This concludes the fact that 
$$
\int_{\R^n} \exp \left( -4b f \left(\theta \right)^2 \right)\exp \left( -|\theta|_2^2 /2 \right) d\theta \le \left(\frac{C t}{\gamma \sqrt b} \right)^n + \exp \left(-b \alpha^2 \right). 
$$
Using Lemma \ref{integ struc}, Theorem \ref{thm lo} follows. 
\end{proof}

\bibliographystyle{amsalpha}
\bibliography{bib-file}

\end{document}